\documentclass[11pt]{amsart}

\usepackage{graphicx}
\usepackage{float}
\usepackage[toc,page]{appendix}
\usepackage{bm}
\usepackage{comment}
\usepackage{fullpage}

\newtheorem{theorem}{Theorem}[section]
\newtheorem{heuristic}[theorem]{Heuristic}
\newtheorem{lemma}[theorem]{Lemma}
\newtheorem{conjecture}[theorem]{Conjecture}
\newtheorem{proposition}[theorem]{Proposition}

\title{Irreducibility of Random Polynomials}

\address{Department of Mathematics, University of Wisconsin - Madison, Madison, Wisconsin 53706}
\author[C. Borst]{Christian Borst}
\author[E. Boyd]{Evan Boyd}
\author[C. Brekken]{Claire Brekken}
\author[S. Solberg]{Samantha Solberg}
\author[M. M. Wood]{Melanie Matchett Wood}
\thanks{Christian Borst, Evan Boyd, Claire Brekken, and Samantha Solberg were supported by NSF grant DMS-1301690}
\thanks{Melanie Matchett Wood was supported by an American Institute of Mathematics Five-Year Fellowship, a Packard Fellowship for Science and Engineering, a Sloan Research Fellowship, and National Science Foundation grant DMS-1301690.}

\author[P. M. Wood]{Philip Matchett Wood}
\email{enboyd@wisc.edu}
\email{sgsolberg@wisc.edu}
\email{cborst@wisc.edu}
\email{cbrekken@wisc.edu}
\email{mmwood@math.wisc.edu}
\email{pmwood@math.wisc.edu}


\begin{document}

\maketitle

\begin{abstract}
We study the probability that a random polynomial with integer coefficients is reducible when factored over the rational numbers.  Using computer-generated data, we investigate a number of different models, including both monic and non-monic polynomials. Our data supports conjectures made by Odlyzko and Poonen and by Konyagin, and we formulate a universality heuristic and new conjectures that connect their work with Hilbert's Irreducibility Theorem and work of van der Waerden.
The data indicates that 
the probability that a random polynomial is reducible divided by the probability that there is a linear factor appears to approach a constant  and, in the large-degree limit, this constant appears to approach one.  In cases where the model makes it impossible for the random polynomial to have a linear factor, the probability of reducibility appears to be close to the probability of having a non-linear, low-degree factor. We also study characteristic polynomials of random matrices with $+1$ and $-1$ entries.
\end{abstract}

\section{Introduction}
Hilbert's Irreducibility Theorem states that a monic polynomial of degree $d$, where each coefficient is chosen uniformly and independently from integers in the interval $[-K,K]$, is irreducible over the integers with probability tending to one as $K$ goes to infinity. This statement of the theorem was proved by van der Waerden~\cite{van der} in 1934. In 1963, Chela \cite{chela} proved that the number of reducible (over the integers) polynomials of this form divided by $K^{d-1}$ approaches a constant as $K$ goes to infinity.  Chela's result \cite{chela} can be interpreted as proving that the probability of reducibility divided by the probability that the constant coefficient equals zero approaches a constant  as $K$ goes to infinity (see Theorem~\ref{Chela Theorem} and Figure~\ref{Low Chela}), and we believe that this interpretation is an example of a universal phenomenon. 

\begin{figure}[t]
    \includegraphics[width=\textwidth,height=.5\textheight,keepaspectratio]{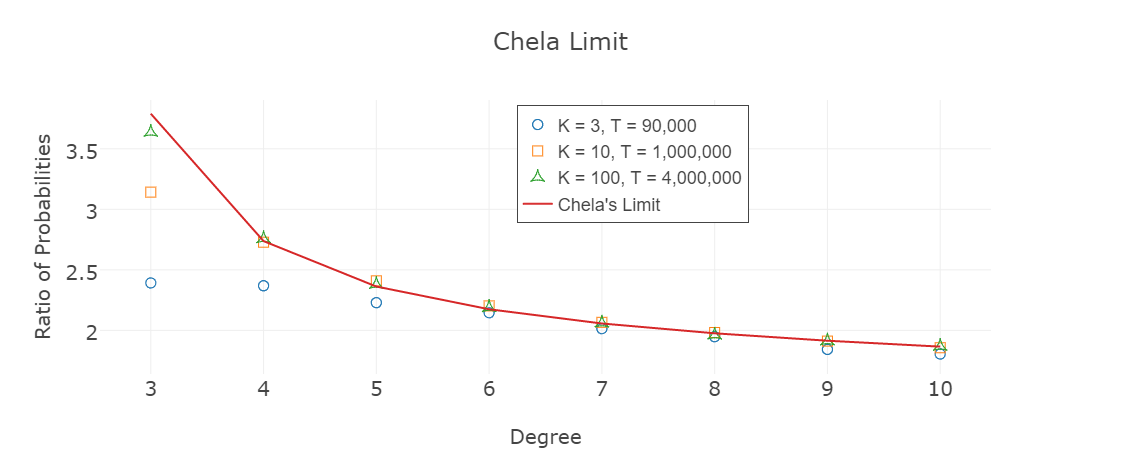}
    \caption{Let $h(x)$ be a degree $d$ monic polynomial with all other coefficients chosen independently and randomly from integers in the interval $[-K,K]$.  A result proven by Chela \cite{chela} (see Theorem \ref{Chela Theorem}) implies that the probability that $h(x)$ is reducible divided by the probability the probability that the constant coefficient is zero goes to a limit as $K$ goes to infinity. Above is a plot of data supporting this result when $K=3,10$ and $100$, where $T$ is the number of random trials trials to compute the respective data points. The data points are a ratio of the probability of reducibility divided by $\frac{1}{2K+1}$, which is the probability that the constant coefficient is zero. The figure indicates that convergence to the limit proven by Chela is relatively fast, especially for degree 6 and larger.
    The $99.9999\%$ confidence intervals are $\pm0.058$, $\pm0.053$, and $\pm0.251$ for $K = 3$, $10$, and $100$, respectively.}
    \label{Low Chela}
\end{figure}

\begin{heuristic}[Universality] \label{universality}
    Let $f(x)$ be a random polynomial with sufficiently well-behaved integer coefficients. Then the probability that $f(x)$ is reducible over the rationals divided by the probability that $f$ has a linear (or lowest possible degree) factor approaches a constant $C$ in the limit as the degree goes to infinity, or in the limit as the support of the random coefficients goes to infinity, or both.  Furthermore, in the limit as the degree goes to infinity, whether or not the support goes to infinity, the constant $C$ should equal 1.
\end{heuristic}

We study many polynomial models that appear to satisfy Heuristic~\ref{universality} (primarily with independent coefficients, though some with dependence, see Section~\ref{Matrix}), and the ``well-behaved'' condition is meant to exclude models with specific features that cause high-degree factors, for example, a random polynomial of degree $d$ formed as the product of two random  polynomials with degree around $d/2$.  Throughout, ``reducible'' will be used as shorthand for  ``reducible over the rationals.''  Note that for monic polynomials, reducibility over the rationals is equivalent to reducibility over the integers by Gauss's Lemma.  For non-monic polynomials, we are interested in cases where all factors have positive degree, so we factor over the rationals; for example, $2x^2+4 = 2(x^2+2)$ is irreducible over the rationals, because $2$ has a multiplicative inverse in the rationals.

O'Rourke and Wood~\cite{orourke} study a variety of models, proving that pointwise delocalization of the roots of a random polynomial implies that low-degree factors are very unlikely.  Konyagin~\cite{konyagin} studied random monic polynomials with 0 or 1 for coefficients and proved a bound on the probability of low-degree factors which he then used to prove a bound on the probability of reducibility.  Heuristic~\ref{universality} compliments these results on low-degree factors by suggesting that high-degree factors are even less likely than low-degree factors, and proving results in this direction is an interesting open problem.

Kuba \cite{kuba} studied non-monic polynomials with integer coefficients chosen independently in the interval $[-K,K]$ (conditioned on the lead coefficient being nonzero), and proved a result implying that the 
probability of reducibility divided by the probability that the constant term is zero is bounded by a constant. This can be viewed as a step towards proving a version of Chela's~\cite{chela} result (and a case of Heuristic~\ref{universality}) for non-monic polynomials. In Section~\ref{Growing Support with Fixed Degree}, we will test generalizations of Hilbert's Irreducibility Theorem, studying polynomials with fixed degree in the limit as the support for the coefficients grows, and we will state special cases of Heuristic~\ref{universality} as conjectures.

In Hilbert's Irreducibility Theorem, the degree is fixed and the support of the coefficients is growing, and we believe that Heuristic~\ref{universality} also holds for cases where the range of support of the coefficients is fixed and the degree goes to infinity. For example, let $g_{\{0,1\},d}(x)$ be a monic degree $d$ polynomial with constant coefficient equal to 1 and with every other coefficient equal to 0 or 1 independently with probability $\frac{1}{2}$. In 1993, Odlyzko and Poonen~\cite{odlyzko} conjectured that the probability of irreducibility for these polynomials approaches 1 as the degree $d$ approaches infinity. Konyagin~\cite{konyagin} then conjectured in 1999 that the probability that $x + 1$ is a factor of $g_{\{0,1\},d}(x)$, given that it is reducible, approaches 1 as $d$ approaches infinity, which can be viewed as a case of Heuristic~\ref{universality} (see Figure~\ref{Z1 Low}).

\begin{figure}
    \includegraphics[width=\textwidth,height=.5\textheight,keepaspectratio]{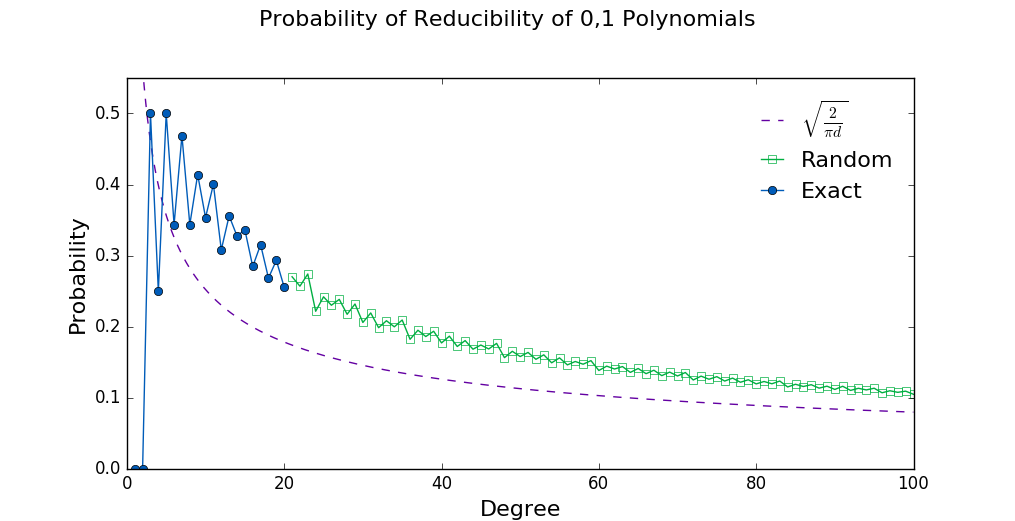}
    \caption{The plot above shows the probability of reducibility for monic $0,1$ polynomials, where the lead and constant coefficients are 1 and all other coefficients are 0 or 1 independently with probability $1/2$ (note this figure also appears in \cite[Figure~1]{orourke}). The exact points (circles) were computed by exhaustive generation of all $0,1$ polynomials of degree up to 20. The random points (squares) were generated with 1,000,000 random trials, giving a 99.9999\% confidence interval of  $\pm 0.0025$. The 
    curve (dashed) is an asymptotic lower bound for the probability of reducibility derived from the probability that $x-1$ is a factor of the polynomial (see Proposition~\ref{P1}).  The data shows that the probability of reducibility approaches the probability that $x-1$ is a factor, supporting Konyagin's conjecture \cite{konyagin} and Heuristic~\ref{universality} (see also Figure~\ref{Z1 High}).}
    \label{Z1 Low}
\end{figure}

We consider the following related model: Let $g_{\pm1,d}(x)$ be a monic polynomial with degree $d$ with all other coefficients $+1$ or $-1$ independently with probability $\frac{1}{2}$. In Conjecture~\ref{op}, we extend Odlyzko and Poonen's conjecture to these polynomials as well. A version of Konyagin's conjecture appears to apply to odd-degree polynomials $g_{\pm1,d}(x)$, where having a linear factor appears to be the most common way for such a polynomial to factor---see Figure~\ref{PM1 Red} and Conjecture~\ref{PM1 Linear}. Even-degree polynomials $g_{\pm1,d}(x)$ cannot have a linear factor, but these polynomials may still support a variant of Konyagin's conjecture where having a quadratic or other low-degree factor appears to be the most common way for such a polynomial to factor---see Lemma~\ref{lemma1}. Also, O'Rourke and Wood \cite{orourke} prove that all polynomials of this form have a vanishingly small probability of having a factor with any fixed degree, including linear factors, as $d$ goes to infinity. Data for both polynomials $g_{\{0,1\},d}(x)$ and $g_{\pm1,d}(x)$ supports these conjectures and our Heuristic~\ref{universality} (see Section~\ref{Fixed Support with Growing Degree}).

\begin{figure}
    \includegraphics[width=\textwidth,height=.5\textheight,keepaspectratio]{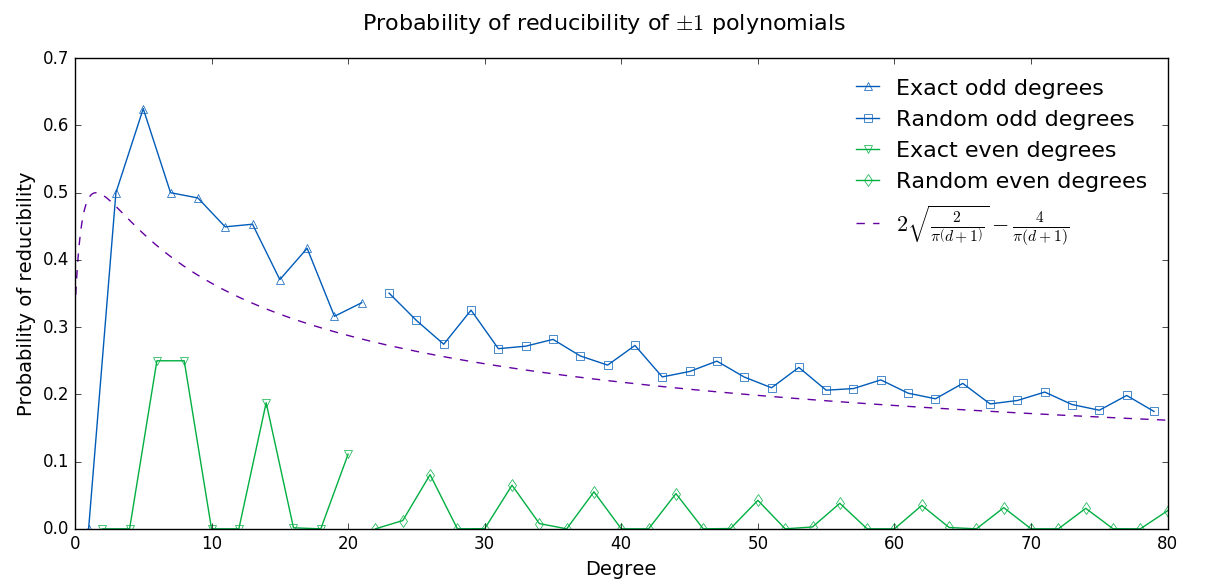}
    \caption{The plot above shows probability of reducibility for monic polynomials where each coefficient is $+1$ or $-1$ independently with probability $1/2$ (note that this figure also appears in \cite[Figure~2]{orourke}). The random data was calculated using 150,000,000 trials, giving a 99.9999\% confidence interval of 0.0002 for every value. The lower bound on the probability of reducibility for odd degrees (dashed) is the asymptotic probability that a degree $d$ polynomial has a linear factor (see Proposition~\ref{P4}).  If the degree is even, the polynomials cannot have a linear factor, and the reducibility probability appears to be distinctly lower.  For even degrees, a Galois theory argument (\cite[Subsection~2.2]{orourke}) proves that the reducibility probability is exactly zero whenever $d+1$ is prime and 2 generates the multiplicative group of $\left(\mathbb Z/(d+1)\right)^\times$, which should happen for infinitely many $d$ by Artin's Conjecture (see, for example, \cite{Moree}).}
    \label{PM1 Red}
\end{figure}

Random polynomials with coefficients $1$ or $-1$ have also been studied by Peled, Sen, and Zeitouni \cite{PSZeitouni2016}, and they prove that the probability of a double root is asymptotically equal to the probability that either $-1$ or $1$ is a double root (and they further extend this result to polynomials with coefficients $1$, $0$, or $-1$).  This result is related to Section~\ref{Fixed Support with Growing Degree} where we conjecture for random polynomials with $1$ or $-1$ coefficients that, among reducible polynomials, the probability that $x+1$ or $x-1$ is a factor tends to 1 (see Conjecture~\ref{PM1 Linear}).  Very roughly, both Peled, Sen, and Zeitouni's result \cite{PSZeitouni2016} and Conjecture~\ref{PM1 Linear} suggest that the coincidence of having a double root or the coincidence of factoring are most likely to happen in the simplest way, and that is when $1$ or $-1$ is a root.

Random polynomials of the sort considered in Hilbert's Irreducibility Theorem have been studied in probabilistic Galois theory. For example, van der Waerden~\cite{van der} proved that if $h_{[-K,K],d}(x)$ is monic with degree $d$ and all of its coefficients are chosen independently and uniformly from integers in the interval $[-K,K]$, then the probability that the Galois group for $h_{[-K,K],d}(x)$ is $S_d$, the symmetric group of $d$ elements, tends to 1 as $K$ tends to infinity. Work proving successively better upper bounds on the probability that $h_{[-K,K],d}(x)$ does not have Galois group $S_d$ has continued with the work of Knobloch in 1955 \cite{knobloch1} and 1956 \cite{knobloch2}, Gallagher in 1973 \cite{gallagher}, Zywina in 2010 \cite{zywina}, Dietmann in 2013 \cite{dietmann}, and finally Rivin \cite{rivin} in 2015, who proved the upper bound $\frac{\log^{c}{K}}{K}$, where $c=c_d$ is a constant depending on $d$, which is the first upper bound demonstrating that the probability decreases linearly in $K$, up to a polylog factor.

For this paper, computer simulations were used to randomly measure data points in the real numbers whose exact values are unknown, and we quantify the likely difference between our measured data points and the actual values using confidence intervals.  For example, we say a randomly measured data value of $x$ has a $99.9999\%$ (five standard deviations) confidence interval of $\pm \epsilon$ to mean that, at the start of the random experiment, there is at most a one-in-a-million chance (i.e., probability $0.000001$) that the measured value will differ from the actual value by more than $\epsilon$. 
The data for this paper was collected using HTCondor~\cite{basney,litzkow,tannenbaum,thain}, a high-throughput computing software system developed and maintained at the University of Wisconsin-Madison, and programming was done primarily in Magma.

\begin{figure}
    \includegraphics[width=\textwidth,height=.5\textheight,keepaspectratio]{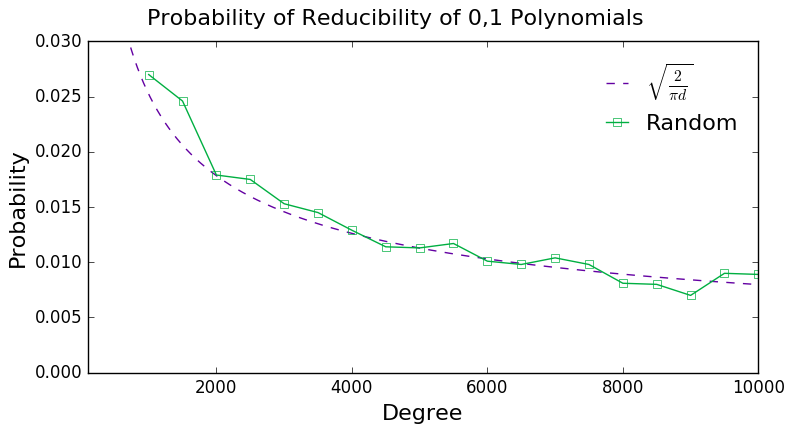}
    \caption{The probability of reducibility for $0,1$ polynomials of high degree. The random data (squares) was generated with 10,000 trials, a relatively small number necessitated by the high degree, giving a 95\% (2 standard deviations) confidence interval of $\pm 0.010$, which is relatively large for this figure. The asymptotic lower bound (dashed) is the same bound as in Figure~\ref{Z1 Low}, and we believe some points fall below this asymptotic lower bound due to measurement error, which would be corrected if substantially more trials were conducted.}
    \label{Z1 High}
\end{figure}

The paper is organized as follows. Section~\ref{Fixed Support with Growing Degree} studies examples of monic polynomials where the support of coefficients is fixed and the degree is increasing, including coefficients either $0$ or $1$ and coefficients either $1$ or $-1$, where each coefficient is independently chosen with probability $\frac{1}{2}$.  Section~\ref{Growing Support with Fixed Degree} gives examples of random monic and non-monic polynomials where the degree is fixed and the range of support of the coefficients is growing.  This includes coefficients chosen uniformly from integers in the interval $[-K,K]$ and other distributions on integers in the interval $[-K,K]$.  Section~\ref{Matrix} studies characteristic polynomials of random $d$ by $d$ matrices where each entry is $+1$ or $-1$ independently with probability $1/2$, a case where where both the degree of the polynomial and the support of the coefficients tend to infinity as $d$ increases.  Propositions, lemmas, and proofs appear in Section~\ref{Asymptotic Lower Bounds}.

\section{Fixed Support with Growing Degree} \label{Fixed Support with Growing Degree}
\subsection*{Zero-One Polynomials}

Let $g_{\{0,1\},d}(x)= x^d+ a_{d-1}x^{d-1} + \dots + a_1 x + 1$ be a random monic polynomial with degree $d$ and constant coefficient 1, where every other coefficient $a_1,\dots, a_{d-1}$ is 0 or 1 independently with probability $1/2$.  As shown in Proposition~\ref{P1}, the probability that $g_{\{0,1\},d}(x)$ has $x + 1$ as a factor is asymptotically $\sqrt{\frac{2}{\pi d}}$, which is a lower bound for the probability of reducibility. Our data suggests that as the degree increases, the probability of reducibility converges to this lower bound (see Figures~\ref{Z1 Low} and \ref{Z1 High}).

Consider a random monic polynomial $h_{\{0,1\},d}(x)$ with the same form as $g_{\{0,1\},d}(x)$ above, but with the constant term also allowed to be 0 or 1 independently with probability $1/2$. One could ask: For polynomials $h_{\{0,1\},d}(x)$, does the probability of reducibility divided by the probability that the constant term is 0 go to 1 as the degree $d$ increases? The answer to this question is yes if and only if the probability that the polynomials $g_{\{0,1\},d}(x)$ are irreducible goes to 1 as the degree $d$ goes to infinity. The equivalence is true because the probability of reducibility for polynomials $h_{\{0,1\},d}(x)$ is equal to $1/2$ (the probability that the constant coefficient is zero) plus the probability of reducibility for polynomials $g_{\{0,1\},d}(x)$.
 Thus, Heuristic~\ref{universality} for $h_{\{0,1\},d}(x)$ is equivalent to Odlyzko-Poonen's conjecture \cite{odlyzko} that $g_{\{0,1\},d}(x)$ becomes irreducible with probability tending to 1 as $d$ goes to infinity.

\subsection*{Rademacher $\bm{\pm1}$ Polynomials}

\begin{figure}
    \includegraphics[width=\textwidth,height=.5\textheight,keepaspectratio]{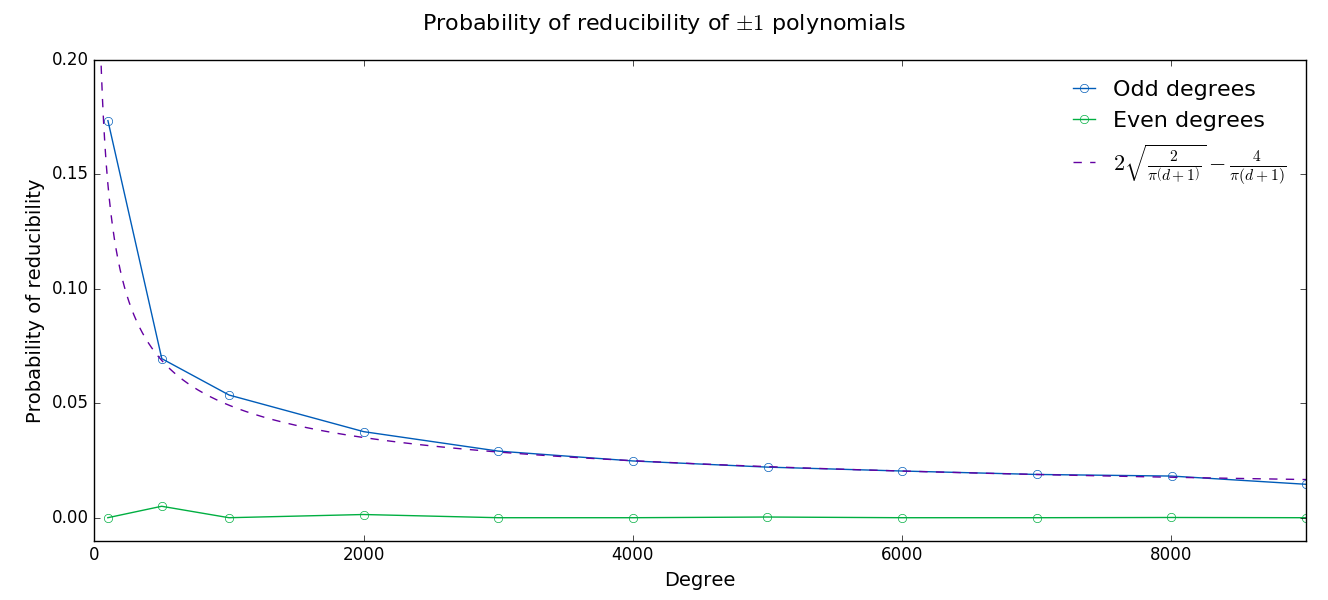}
    \caption{The probability of reducibility for $\pm 1$ polynomials with high degree. The data was calculated randomly using 10,000 trials, giving 95\% (2 standard deviations) confidence interval of $\pm 0.01$.
    The lower bound for odd-degrees is the same as the lower bound in Figure~\ref{PM1 Red}.}
    \label{PM1 Large}
\end{figure}

Let $g_{\pm1,d}(x)= x^d + a_{d-1}x^{d-1} + \dots + a_0$ be a monic, degree $d$ polynomial where each coefficient $a_0,a_1,\dots,a_{d-1}$ is either 1 or $-1$ independently with probability $\frac{1}{2}$.
 In Figure~\ref{PM1 Red} and Figure~\ref{PM1 Large}, we see that the probability of reducibility of $g_{\pm1,d}(x)$ polynomials decreases as the degree $d$ goes to infinity. 
For odd-degree $d$, we prove that 
$2 \sqrt{\frac{2}{\pi(d+1)}} - \frac{4}{\pi(d+1)}$ is an asymptotic lower bound, based on the probability that $x+1$ or $x-1$ is a factor (see Lemma~\ref{lemma1} and Proposition~\ref{P4}).  For even-degree $d$, a linear factor over the integers is impossible (see Lemma~\ref{lemma1}).  Figure~\ref{PM1 Red} and Figure~\ref{PM1 Large} suggest that, when the degree is odd, the probability that $g_{\pm1,d}(x)$ is reducible approaches the probability that $g_{\pm1,d}(x)$ has a linear factor, supporting Heuristic~\ref{universality} and suggesting the following conjecture, which is analogous to Konyagin's conjecture in \cite{konyagin}.

\begin{conjecture}\label{PM1 Linear}
    Let $d$ be an odd positive integer.  For random monic polynomials $g_{\pm1,d}(x) = x^d + a_{d-1}x^{d-1} + \dots + a_0,$ where the $a_i$ are $+1$ or $-1$ independently with probability $1/2$, the probability that $g_{\pm1,d}(x)$ has a linear factor $x + 1$ or $x - 1$, conditioned on $g_{\pm1,d}(x)$ being reducible, goes to 1 as the degree $d$ goes to infinity.
\end{conjecture}

The data also suggests the following analog of Odlyzko and Poonen's~\cite{odlyzko} conjecture.

\begin{conjecture}\label{op}
 Let $d$ be a  positive integer.  For random monic polynomials $g_{\pm1,d}(x) = x^d + a_{d-1}x^{d-1} + \dots + a_0,$ where the $a_i$ are $+1$ or $-1$ independently with probability $1/2$, the probability that $g_{\pm1,d}(x)$ is reducible goes to 0 as the degree $d$ goes to infinity.
\end{conjecture}

 



\begin{figure}
    \includegraphics[width=\textwidth,height=.6\textheight,keepaspectratio]{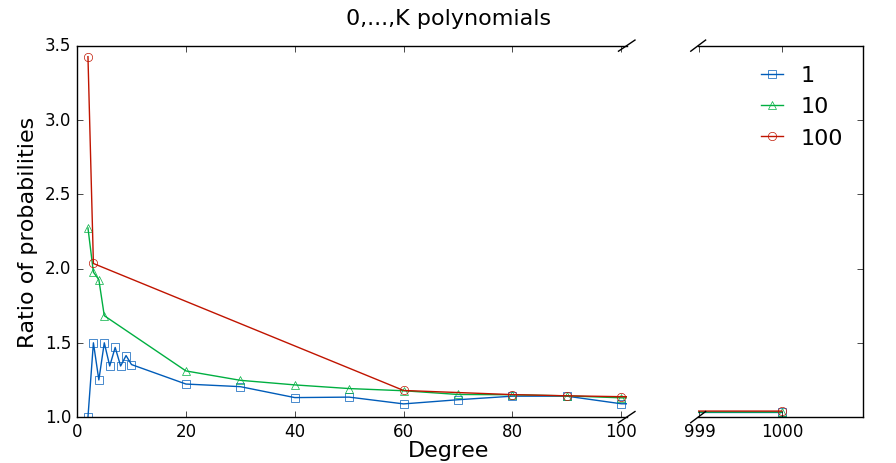}
    \caption{Consider a random monic polynomial where each coefficient is sampled uniformly from integers in the interval $[0,K]$.  Above is a plot of data measuring the probability that such a polynomial is reducible divided by the probability that the constant term is 0, which is $\frac{1}{K + 1}$. The ratio appears to tend to 1 as the degree increases, supporting Heuristic~\ref{universality}. Each data point was generated with $2500K^2$ trials, giving 99.9999\% confidence intervals of $\pm 0.1$ , $\pm 0.055$ and $\pm 0.0505$, respectively, for $K=1$, 10, and $100$.}
    \label{ZeroK}
\end{figure}

\section{Growing Support with Fixed Degree} \label{Growing Support with Fixed Degree}
\subsection*{Uniform integers in the interval $[-K,K]$}
Chela~\cite{chela} proved the theorem below, which we have rephrased in terms of a ratio of probabilities to show its connection with Heuristic~\ref{universality}.

\begin{theorem}[Chela~\cite{chela}]\label{Chela Theorem}
Let $h=h_{[-K,K],d}(x) =x^d + a_{d-1}x^{d-1} + \dots + a_1x + a_0$ be a polynomial with degree $d> 2$, where the $a_i$ are chosen uniformly and independently from integers in the interval $[-K,K]$. As $K$ goes to infinity, the probability of reducibility divided by the probability that the constant coefficient is zero goes to the limit $C_d$, i.e., 
\begin{equation}\label{e:chela}
\lim_{K\to\infty} \frac{\Pr(h(x) \mbox{ is reducible})}{ \frac{1}{2K+1} } = C_d :=  2 \zeta(d-1) -1  + \frac{k_{d}}{2^{d-2}},
\end{equation}
where $k_d = \int \int \dots \int \,dx_1\, dx_2 \,\dots\, dx_{d-1}$, with the iterated integral running over all $x_i$ satisfying $|x_i| \le 1\mbox{ for } 1 \le i \le d-1$ and $\left| x_1 + x_2 + \dots +x_{d-1} \right| \le 1$.
\end{theorem}

We generated data for polynomials $h_{[-K,K],d}(x)$ to find the probability of reducibility for varying K values---see Figure~\ref{Low Chela}. It is interesting to note that our data fits closely with Chela's limit for small $K$ values, especially when $d$ increases. Also note that the right-hand side of \eqref{e:chela} in Theorem~\ref{Chela Theorem} approaches $1$ as $d$ goes to infinity, supporting Heuristic~\ref{universality}.

\subsection*{Uniform integers in the interval $[0,K]$}
We consider the probability of reducibility for polynomials of the form $f_{[0,K],d}(x) = x^d + a_{d-1}x^{d-1} + \dots + a_1x + a_0$, where the $a_i$ are chosen uniformly and independently from the interval $[0,K]$ and K is a positive integer. To study Heuristic~\ref{universality} in this case, we collected data for the probability of reducibility divided by the probability that the constant coefficient is zero  (which is $1/(K+1)$) for various degrees and values of $K$.  Figure~\ref{ZeroK} indicates that, as predicted by Heuristic~\ref{universality}, the ratio of probabilities appears to tend to a constant as $K$ increases and to tend to 1 as the degree increases.


\begin{figure}
    \includegraphics[width=\textwidth,height=.6\textheight,keepaspectratio]{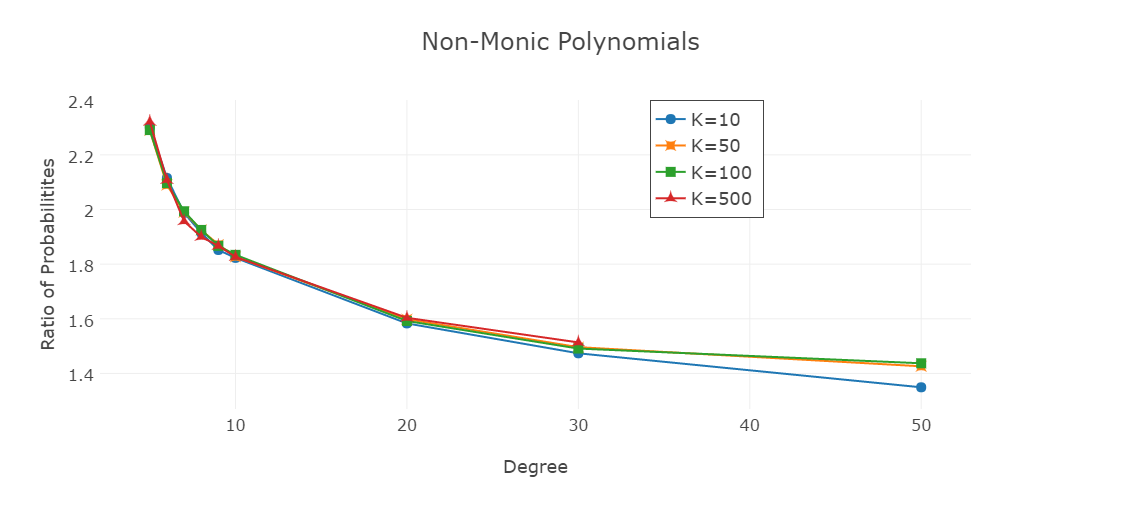}
    \caption{Above is a plot of data supporting a result due to Kuba \cite{kuba} (Theorem~\ref{t:kuba}) for $K=10,50,100,$ and $500$, with the data generated using 
    250,000, 6,250,000, 25,000,000, and 625,000,000 trials, respectively. 
     The points are a ratio of the probability of reducibility divided by $\frac{1}{2K+1}$, which is the probability that the constant coefficient is zero. 
     %
    The $99.9999\%$ confidence interval for each data point is at most $\pm 0.105$.}
    \label{Kuba's Power}
\end{figure}

\subsection*{Non-Monic Uniform Coefficients}

Kuba \cite{kuba} considered non-monic polynomials and proved the following theorem, which we have rephrased in terms of a ratio of probabilities to show its connection with Heuristic~\ref{universality}.

\begin{theorem}[Kuba \cite{kuba}] \label{t:kuba}
Let $h_{K,d}(x) = a_dx^d + a_{d-1}x^{d-1}+\dots+a_0$ be a polynomial  with degree $d > 2$, where the $a_i$ are chosen uniformly and independently from integers in the interval $[-K,K]$, and $a_d$ is nonzero. Then, 
$$\frac{\Pr( h_{K,d}(x) \mbox{ is reducible})}{\frac{1}{2K+1}} \le C_d,\qquad \mbox{for every $K\ge 1$},$$
where $C_d$ is a constant depending only on $d$. (Note that the left-hand side of the inequality above is always at least 1.)
\end{theorem}
Our computer data indicates that the probability that $h_{K,d}(x)$ is reducible divided by $\frac{1}{2K+1}$ decreases to 1 as $d$ goes to infinity; see Figure~\ref{Kuba's Power}. Data with higher degree (we also tested $d = 80,90,100,200$) further suggests that this ratio goes to 1.

\subsection*{Binomial Coefficients}

One natural, non-uniform way to choose coefficients in the interval $[0,K]$ is to use a binomial distribution, $\operatorname{Bin}(K,p)$, where $K$ is a positive integer and $0\le p\le 1$, so that the value $\ell$ is taken with probability $\binom{K}{\ell}p^\ell(1-p)^{K-\ell}$.  Heuristic~\ref{universality} is easiest to understand when the probability of a linear factor is dominated by the probability that the constant coefficient is zero, and thus we let $p = 1/K$, in which case the constant coefficient is zero with probability $(1-1/K)^K$, which is close to $ e^{-1}$ especially as $K$ increases.  Interestingly, this binomial model has features similar to 0,1 polynomials, in that each coefficient takes the value zero with roughly constant probability as $K$ increases, while still having the support of the coefficients tend to infinity as $K$ increases (the binomial distribution approximates a Poisson distribution as $K$ increases).  Data in Figure~\ref{Binomial} suggests that Heuristic~\ref{universality} still holds, and we make the following conjecture.

\begin{conjecture} \label{Binomial Conjecture}
    For a random monic polynomial $h_{\mathrm{Bin},K,d}(x)=x^d+a_{d-1}x^{d-1}+\dots+a_1 x+a_0$ with coefficients $a_0,\dots, a_{d-1}$ chosen from the interval $[0,K]$ using a $\operatorname{Bin}(K,1/K)$ distribution, the probability that the polynomial is reducible divided by the probability that the constant coefficient is zero goes to 1 as $K$ goes to infinity.
\end{conjecture}

\begin{figure}
    \includegraphics[width=\textwidth,height=.5\textheight,keepaspectratio]{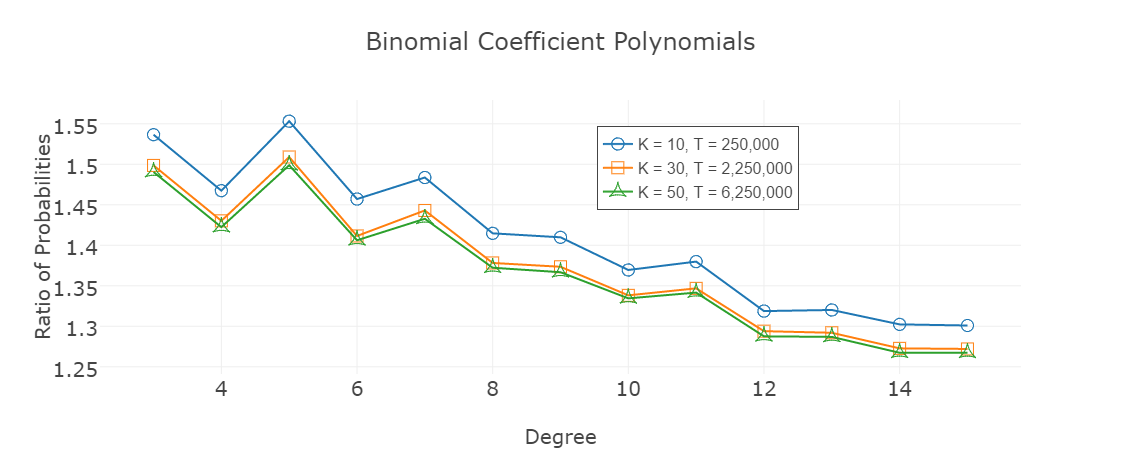}
    \caption{Let $h_{{\rm Bin},K,d}(x)$ be a monic random polynomial with all other coefficients chosen independently from $\{0,1,\dots, K\}$ using a $\operatorname{Bin}(K,1/K)$ distribution.  The plot shows the probability that $h_{{\rm Bin},K}(x)$ is reducible divided by the probability that the constant coefficient is zero for degrees $d = 3$ to $d = 15$ and $K$ equal to 10, 30, and 50.  Note that the ratio appears to tend to $1$ as the degree increases, supporting Heuristic~\ref{universality} (see Conjecture~\ref{Binomial Conjecture}).  For the data points for $K=10$, 30, and 50, the 99.9999\% confidence intervals are $\pm0.0017$, $\pm0.0006$ and $\pm0.0004$, respectively.  }
    \label{Binomial}
\end{figure}

\section{Characteristic polynomial of a random $\pm$1 matrix.} \label{Matrix}

We consider the probability of reducibility for $\chi(x)$, the characteristic polynomial of a $d$ by $d$ matrix with integer entries.
This case is different than the previous two sections because both the degree of the characteristic polynomial and the support of the coefficients grow as $d$ increases.  We study the case where the $d$ by $d$ matrix has entries that are $+1$ or $-1$ independently with probability $\frac{1}{2}$.  Some results are known for other models of random matrices.  In \cite{Rivin2008}, Rivin studies random integer matrices formed as products of random elements in the special linear group $\operatorname{SL}(n,\mathbb Z)$, where the random elements are chosen in two ways: either uniformly over all matrices in $\operatorname{SL}(n,\mathbb Z)$ with coefficients uniformly bounded by a constant, or in terms of word length based on a fixed generating set for $\operatorname{SL}(n,\mathbb Z)$.  Rivin \cite{Rivin2008} shows that the probability that the characteristic polynomial is reducible tends to 0 in both cases as $n$ increases, and also studies $\operatorname{Sp}(n,\mathbb Z)$ (see further extensions in \cite{GNevo2011}).  

In the case where the $d$ by $d$ matrix has entries that are $+1$ or $-1$ independently with probability $\frac{1}{2}$, Figure~\ref{Matrix Plot} indicates that the probability of reducibility decreases as $d$ increases, and the probability of reducibility appears to become close to a lower bound derived from the probability that the matrix has a pair or rows or columns that are dependent (this implies that the constant coefficient, which is the determinant of the matrix, is zero).  It has long been conjectured (see, for example, \cite{komlos,kahn, tao1, tao2, bourgain, AS}) that the probability that a uniform random $d$ by $d$ matrix with $\pm 1$ entries has determinant zero is asymptotically \begin{equation}\label{singbd}
4 \binom{d}{2} \left(\frac12\right)^d =\left(\frac12 + o(1)\right )^d,
\end{equation} where $o(1)$ is a small quantity tending to zero as $d$ tends to infinity.  The  term $4 \binom{d}{2} \left(\frac12\right)^d$ is an asymptotic lower bound coming from the probability that there is a pair of rows or columns that are dependent.  Figure~\ref{Matrix Plot} uses a different asymptotic lower bound (see Proposition~\ref{P5}) that is similar but more accurate for small $d$.  Based on our data we make the following conjecture, which implies the bound in \eqref{singbd} and also supports Heuristic~\ref{universality}.

\begin{conjecture}\label{CharPolyIrred}
If $\chi(x)$ is the characteristic polynomial of a $d$ by $d$ random matrix, where each entry is $+1$ or $-1$ independently with probability $1/2$, then
$$\lim_{d\to\infty} \frac{\Pr( \chi(x) \mbox{ is reducible})}
{4 \binom{d}{2} \left(\frac12\right)^d }
=
1.$$
\end{conjecture}

\begin{figure}
    \centering
    \includegraphics[keepaspectratio, width=1\textwidth]{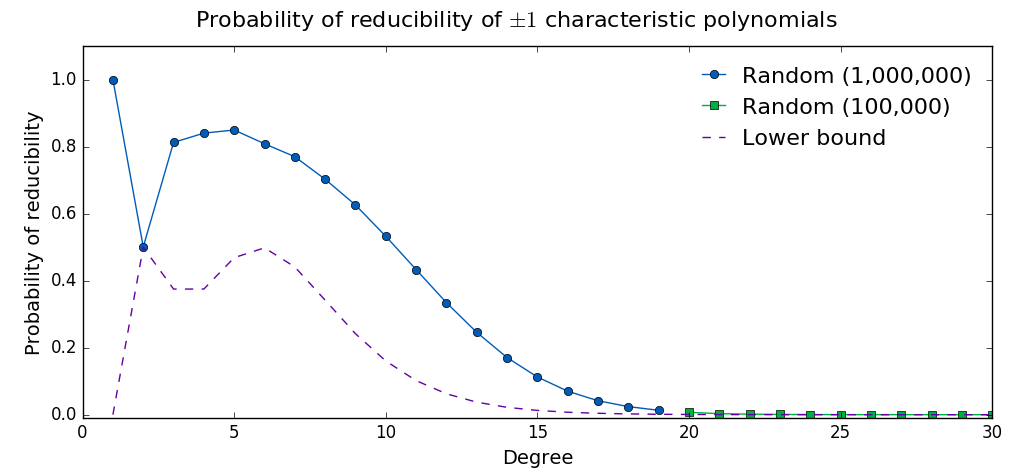}
    \caption{Above is a plot of the probability of reducibility for the characteristic polynomial of a random $d$ by $d$ matrix with entries either $+1$ or $-1$ independently with probability $\frac{1}{2}$. The first $19$ points (circles) were calculated randomly using 1,000,000 trials, giving a 99.9999\% confidence interval of $\pm 0.0025$ for each value. The remaining points (squares) were calculated randomly using 100,000 trials, giving a 99.9999\% confidence interval of $\pm 0.0079$. The asymptotic lower bound curve (dashed line) is from Proposition~\ref{P5}.}
    \label{Matrix Plot}
\end{figure}

\section{Propositions and proofs}\label{Asymptotic Lower Bounds}
In this section we collect lower bounds on the probability of a linear factor in given polynomial models using combinatorics. We use the notation $o(1)$ to signify a small function which tends to 0 as $d$ tends to infinity.
\begin{proposition}\label{P1}
For a random monic polynomial $g_{\{0,1\},d}(x) = x^d + a_{d-1}x^{d-1} + \dots + a_1x + 1$ where $a_i$ is 0 or 1 independently with probability $\frac{1}{2}$, the probability that $g_{\{0,1\},d}(x)$ has a linear factor is $\left(1 + o(1)\right)\sqrt{\frac{2}{\pi d}}$.
\end{proposition}

\newcommand\floor[1]{\left\lfloor #1 \right\rfloor}
\newcommand\ceiling[1]{\left\lceil #1 \right\rceil}

\begin{proof}
By a geometric series argument similar to \cite[Lemma 4.1]{orourke}, $x + 1$ is the only possible linear factor for a polynomial $g_{\{0,1\},d}(x)$, and to have $x+1$ as a factor, $-1$ must be a root.  We will consider the cases of $d$ odd and $d$ even separately.

If $d$ is odd, then among the random coefficients $a_i$, there are $\frac{d-1}{2}$  even-degree terms and $\frac{d-1}{2}$ odd-degree terms, and the number of even-degree terms with a non-zero coefficient must equal the number of odd-degree terms with a non-zero coefficient; thus, the number of polynomials with $x + 1$ as a factor is when $d$ is even is

$$\sum_{k=0}^{\frac{d-1}{2}} \binom{\frac{d-1}{2}}{k}  \binom{\frac{d-1}{2}}{k}.$$ Using Vandermonde's identity and  Stirling's approximation $k! = \sqrt{2\pi k}(\frac{k}{e})^d \left(1 + o(1)\right)$, we have 

$$\sum_{k=0}^{\frac{d-1}{2}} \binom{\frac{d-1}{2}}{k}  \binom{\frac{d-1}{2}}{\frac{d-1}{2} - k} = \binom{d-1}{\frac{d-1}{2}} = 2^{d-1}\sqrt{\frac{2}{\pi d}}(1 + o(1)).$$ Thus, the probability of having $x + 1$ as a factor is $\sqrt{\frac{2}{\pi d}} (1+o(1))$.

When $d$ is even, the $x^d$ term and the constant term are both positive and thus, if $k$ of the $\frac d2-1$ random coefficients for even-degree terms are non-zero, then $k+2$ of the $\frac d2$ random coefficients for odd-degree terms must also be non-zero.  Thus the number of polynomials with $x+1$ as a factor when $d$ is odd is
$$\sum_{k=0}^{\frac d2 -2} \binom{\frac d2}{k+2} \binom{\frac d2 -1}{k} = \binom{d-1}{\frac d2 -2}$$
by Vandermonde's identity.  Applying Stirling's approximation we have that the probability of having $x+1$ as a factor
$$\frac{1}{2^{d-1}}\binom{d-1}{\frac d2 -2}=\left(\frac{1}{2^{d-1}}\right)\frac{d-2}{8(d+1)}\binom{d+2}{\frac{d+2}{2}} =
\frac{d-2}{d+1}\sqrt{\frac{2}{\pi(d+2)}}(1+o(1)) = \sqrt{\frac{2}{\pi d}} (1+o(1)),$$
completing the proof.  \end{proof}

\begin{lemma}\label{lemma1}
Consider monic polynomials of the form $g(x) = x^d + a_{d-1}x^{d-1} + \dots + a_1x + a_0$, where $a_i$ is $1$ or $-1$. The only possible linear factors over the integers of such a polynomial are $x+1$ and $x-1$.  If $d$ is even, $g(x)$ cannot have a linear factor. If $d$ is odd, the number of polynomials $g(x)$ having $x - 1$ as a factor is equal to the number of polynomials $g(x)$ having $x + 1$ as a factor, which is equal to  $\binom{d}{\frac{d+1}{2}}$.
\end{lemma}
\begin{proof}
First, we prove that there cannot be a linear factor when $d$ is even, using the fact that having a linear factor implies that there is an integer root.  A geometric series argument shows (see \cite[Lemma 4.1]{orourke}) that for a polynomial $g(x)$, all roots must have absolute value strictly between $\frac{1}{2}$ and 2; thus $1$ and $-1$ are the only possible integer roots for a polynomial $g(x)$. In order to have $1$ or $-1$ as a root, the polynomial must have an equal number of positive and negative terms. When $d$ is even, a degree $d$ polynomial has $d + 1$ terms where $d+1$ is odd, and thus, a linear factor is not possible.

Next, we compute the number of polynomials $g(x)$ with $x-1$ as a factor when $d$ is odd.  Having $x-1$ as a factor means that $g(1)=0$.  Note that $g(x)$ has $d+1$ terms and the leading coefficient is 1, and so of the remaining $d$ terms, there must be enough negative coefficients to sum to 0 when $x=1$.  Thus, there must be 
$\frac{d+1}{2}$ negative coefficients in the remaining $d$ terms.  There are $\binom{d}{\frac{d+1}{2}}$ ways to choose which terms will have a negative sign, so there are $\binom{d}{\frac{d+1}{2}}$ degree $d$ polynomials $g(x)$ with $x - 1$ as a factor when $d$ is odd. 

Finally, we use a bijective proof to show that the number of $g(x)$ with $x-1$ as a factor is equal to the number of $g(x)$ with $x+1$ as a factor.  Let $\tilde g(x) = x^d + a_{d-1}x^{d-1} + \dots + a_1x + a_0$ where $a_i$ is $1$ or $-1$ such that $\tilde g(x)$ is a polynomial with $x = 1$ as a root. Then $g(x) = \sum_{i = 0}^d (-1)^{i+1}x^ia_i$ has $x = -1$ as a root. The mapping of $\tilde g$ to $g$ is bijective, thus completing the proof. \end{proof}

\begin{proposition}\label{P4}
Let $d$ be an odd positive integer.  For a random monic polynomial $g_{\pm 1,d}(x) = x^d + a_{d-1}x^{d-1} + \dots + a_0$ where each $a_i$ is $+1$ or $-1$ independently with probability $1/2$, the probability that $g_{\pm 1, d}(x)$ has a linear factor is
$ \left(2 \sqrt{\frac{2}{\pi(d+1)}} - \frac{4}{\pi(d+1)} \right)(1+o(1)).$

\end{proposition}
\begin{proof}

Using Lemma~\ref{lemma1}, we see that an upper bound for the number of polynomials with $1$ or with $-1$ as roots---double counting polynomials with both $1$ and $-1$ as roots---is $2\binom{d}{\frac{d+1}{2}}= \binom{d+1}{\frac{d+1}{2}}.$  
We will use inclusion-exclusion to correct the double counting.  If both $1$ and $-1$ are roots, then $1 + \sum_{i = 0}^{d-1}a_i = 0$ and $-1 + \sum_{i = 0}^{d-1}(-1)^ia_i = 0$. Combining these two equations shows that $1$ and $-1$ are both roots if and only if 
$$\mathop{\sum_{i = 0,}^{d-1}}_{i\ {\rm even}}a_i = 0 \quad \mbox{and}\quad 1 + \mathop{\sum_{i = 0,}^{d-1}}_{i\ {\rm odd}}a_i = 0.$$  This implies that $\frac{d+1}{2}$ is even, so $d\equiv 3 \pmod 4$.  The number of ways to choose such $a_i\in\{+1,-1\}$ is $\binom{\frac{d+1}{2}}{\frac{d+1}{4}}\binom{\frac{d-1}{2}}{\frac{d+1}{4}}= \frac12\binom{\frac{d+1}{2}}{\frac{d+1}{4}}^2$. By inclusion-exclusion, we thus see that the number of polynomials having $+1$ or $-1$ as a root (with no double counting) is
$$\binom{d+1}{\frac{d+1}{2}}- \frac12\binom{\frac{d+1}{2}}{\frac{d+1}{4}}^2 = \left(2 \sqrt{\frac{2}{\pi(d+1)}} - \frac{4}{\pi(d+1)} \right)2^d(1+o(1)),$$
where the second equality is from Stirling's approximation, which implies that $\binom{k}{\frac k2} = \sqrt{\frac{2}{\pi k}}2^{k}(1+o(1))$ for positive even integers $k$.  Dividing by $2^d$ shows that the probability that $g_{\pm 1,d}(x)$ has a linear factor is
$ \left(2 \sqrt{\frac{2}{\pi(d+1)}} - \frac{4}{\pi(d+1)} \right)(1+o(1)).$
\end{proof}

%
%

We end the section with a proposition and proof giving an asymptotic lower bound on the probability that the characteristic polynomial of a $d$ by $d$ matrix with $+1$ and $-1$ entries is reducible.  Note that one could use \cite[Theorem~2.1]{AS} to find a similar lower bound that holds for all $d$; the bound below  highlights the terms that are relevant for small $d$ and is asymptotically accurate for large $d$.

\begin{proposition}\label{P5}
For a characteristic polynomial $\chi(x)$ of a $d$ by $d$ matrix whose entries are $+1$ or $-1$ independently with probability $\frac{1}{2}$, the probability of reducibility is at least $$4\binom{d}{2}\left(\frac{1}{2}\right)^{d} - 2\binom{d}{2}^2 \left(\frac{1}{2}\right)^d \left(\frac{1}{2}\right)^{d-2}-o(1/2^d).$$
\end{proposition}

\begin{proof}
 The determinant of a matrix is the constant coefficient of its characteristic polynomial, and, if the determinant is 0, then $\chi(x)$ has 0 for a root. The above asymptotic lower bound on the probability of reducibility comes from an asymptotic lower bound on the probability the matrix has a determinant $0$. The first term is the asymptotic probability of 2 rows of the matrix being linearly dependent \textit{or} 2 columns being dependent, double counting when both events occur (note that an additional factor of two appears because there are two ways for vectors with all entries in $\{+1,-1\}$ to be dependent).  The second term is the asymptotic probability of 2 rows \textit{and} 2 columns being dependent at the same time.  Using inclusion-exclusion, we subtract these terms to avoid double counting.  The $o(1/2^d)$ accounts for lower-order terms, for example the probability that two or more distinct pairs of rows each have a dependency.  Note we have included the a second-order term even though it is smaller than the error because it makes the bound more accurate for small $d$.
\end{proof}

\section{Acknowledgements}
We thank Steve Goldstein for his support and guidance in programming and collecting data, and we thank the HTCondor system \cite{basney, litzkow, tannenbaum, thain} at the University of Wisconsin-Madison for making our computations possible.  We also thank the Simons Foundation for providing licenses for Magma, the primary computer algebra system used for this project.

\bibliographystyle{amsplain}

\end{document}